\documentclass[11pt]{article}
\usepackage{amsmath,amssymb}

\newtheorem{propo}{{\bf Proposition}}[section]
\newtheorem{coro}[propo]{{\bf Corollary}}
\newtheorem{lemma}[propo]{{\bf Lemma}} \newtheorem{theor}[propo]{{\bf
Theorem}} \newtheorem{ex}{{\sc Example}}[section]

\newenvironment{proof}{{\bf Proof.}}{$\Box$}

\begin{document}

\vspace*{1.0in}

\begin{center} ON CERTAIN DECOMPOSITIONS OF SOLVABLE LIE ALGEBRAS  
\end{center}
\bigskip

\begin{center} DAVID A. TOWERS 
\end{center}
\bigskip

\begin{center} Department of Mathematics and Statistics

Lancaster University

Lancaster LA1 4YF

England

d.towers@lancaster.ac.uk 
\end{center}
\bigskip

\begin{abstract}
The author has previously shown that solvable Lie $A$-algebras and complemented solvable Lie algebras decompose as a vector space direct sum of abelian subalgebras, and their ideals relate nicely to this decomposition. However, neither of these classes is contained in the other. The object of this paper is to find a larger class of algebras, containing each of these classes, in which these same results hold. 
\medskip
 
\noindent {\em Mathematics Subject Classification 2000}: 17B05, 17B20, 17B30, 17B50.
\par
\noindent {\em Key Words and Phrases}: Lie algebras, c-supplemented subalgebras, completely factorisable algebras, Frattini ideal, subalgebras of codimension one. 
\end{abstract}
\section{Introduction}
Throughout $L$ will denote a finite-dimensional solvable Lie algebra over a field $F$. Then $L$ is an {\em $A$-algebra} if all of its nilpotent subalgebras are abelian. We say that $L$ is {\em complemented} if, for any subalgebra $S$ of $L$, there is a subalgebra $T$ such that $S \cap T = 0$ and $\langle S,T \rangle = L$. In \cite{Aalg} and \cite{comp} it was shown that the algebras in each of these classes decompose as a vector space direct sum of abelian subalgebras, and their ideals relate nicely to this decomposition.
\par

 More precisely, it was shown that they split over each term in their derived series. This leads to a decomposition of $L$ as $L = A_{n} \dot{+} A_{n-1} \dot{+} \ldots \dot{+} A_0$ where $A_i$ is an abelian subalgebra of $L$ and $L^{(i)} = A_{n} \dot{+} A_{n-1} \dot{+} \ldots \dot{+} A_{i}$ for each $0 \leq i \leq n$. It is shown that the ideals of $L$ relate nicely to this decomposition: if $K$ is an ideal of $L$ then $K = (K \cap A_n) \dot{+} (K \cap A_{n-1}) \dot{+} \ldots \dot{+} (K \cap A_0)$. However, \cite[Examples 3.8 and 3.9]{comp} showed that neither of these classes was contained in the other. It is natural to ask, therefore, if we can find a larger class of algebras, containing each of these classes, in which these same results hold. We show that the class $q{\mathcal A}$ defined in Section 3 satisfies this. 
\par

We define the {\em nilpotent residual}, $L^{\infty}$, of $L$ be the smallest ideal of $L$ such that $L/L^{\infty}$ is nilpotent. Clearly this is the intersection of the terms of the lower central series for $L$. Then the {\em lower nilpotent series} for $L$ is the sequence of ideals $L_i$ of $L$ defined by $L_0 = L$, $L_{i+1} = (L_i)^{\infty}$ for $i \geq 0$. The {\em derived series} for $L$ is the sequence of ideals $L^{(i)}$ of $L$ defined by $L^{(0)} = L$, $L^{(i+1)} = [L^{(i)},L^{(i)}]$ for $i \geq 0$; we will also write $L^2$ for $L^{(1)}$. If $L^{(n)} = 0$ but $L^{(n-1)} \neq 0$ we say that that $L$ has {\em derived length} $n$. 
\par

We shall denote the nilradical of $L$ by $N(L)$ and the centre of $L$ by $Z(L)$. If $A$ is a subalgebra of $L$ and $B$ is an ideal of $A$, the {\em centraliser} of $A/B$ in $L$ is the set $Z_L(A/B) = \{ x \in L : [x,A] \subseteq B \}$. The {\em Frattini subalgebra} of $L$, $\phi(L)$, is the intersection of the maximal subalgebras of $L$. When $L$ is solvable this is always an ideal of $L$, by \cite[Lemma 3.4]{bg}. Algebra direct sums will be denoted by $\oplus$, whereas direct sums of the vector space structure alone will be denoted by $\dot{+}$.

\section{Splitting over the derived series}
Let ${\mathcal S}_n$ denote the class of solvable Lie algebras that split over $L^{(n)}$. A class ${\mathcal H}$ of finite-dimensional solvable Lie algebras is called a {\em homomorph}
if it contains, along with an algebra $L$, all epimorphic images of $L$.

\begin{lemma}\label{l:hom} ${\mathcal S_n}$ is a homomorph.
\end{lemma}
\begin{proof} This is straightforward.
\end{proof}
\bigskip

Then we have the following properties of  ${\mathcal S}_n$.

\begin{theor}\label{t:split} Let $L$ be a solvable Lie algebra. Consider the following properties.
\begin{itemize}
\item[(i)] $L  \in  {\mathcal S}_n$;
\item[(ii)] $[L^{(n)}, L^{(n-1)}] = L^{(n)}$; and
\item[(iii)] the Cartan subalgebras of $L^{(n-1)}/L^{(n+1)}$ are precisely the complements in $L^{(n-1)}/L^{(n+1)}$ of $L^{(n)}/L^{(n+1)}$ .
\end{itemize}
Then $(i) \Rightarrow (ii) \Rightarrow (iii)$. Moreover, if $L^{(n+1)} = 0$ then all three statements are equivalent.
\end{theor}
\begin{proof} (i) $\Rightarrow$ (ii): Suppose that $L = L^{(n)} \dot{+} B$ for some subalgebra $B$ of $L$. Then $L^{(n-1)} = L^{(n)} \dot{+} B^{(n-1)}$, so
\[
[L^{(n-1)}, L^{(n)}] = L^{(n+1)} + [B^{(n-1)}, L^{(n)}] = [L^{(n)} + B^{(n-1)}, L^{(n)} + B^{(n-1)}] = L^{(n)},
\]
since $B^{(n)} \subseteq L^{(n)} \cap B = 0$.
\medskip

\noindent (ii) $\Rightarrow$ (iii):  From (i) we have that $(L^{(n-1)})^{\infty} = L^{(n)}$, so the result follows from \cite[Theorem 4.4.1.1]{wint}.
\medskip

So suppose now that $L^{(n+1)} = 0$ and that (iii) holds.  Let $C$ be a Cartan subalgebra of $L^{(n-1)}$ and let $L = {\mathcal L}_0 \dot{+} {\mathcal L}_1$ be the Fitting decomposition of $L$ relative to ad\,$C$. Then ${\mathcal L}_1 = \cap_{k=1}^{\infty} L({\rm ad}C)^k \subseteq L^{(n)}$ which is abelian, and so ${\mathcal L}_1$ is an ideal of $L$. Also $L^{(n-1)} = {\mathcal L}_1 \dot{+} {\mathcal L}_0 \cap L^{(n-1)} = {\mathcal L}_1 \dot{+} C$. Now $C$ is abelian, by (iii). But $C \cong L^{(n-1)}/{\mathcal L}_1$, so $L^{(n)} \subseteq {\mathcal L}_1$, whence ${\mathcal L}_1 = L^{(n)}$ and $L = L^{(n)} \dot{+} {\mathcal L}_0$.
\end{proof}
\bigskip

Let ${\mathcal S}_{\infty}$ denote the class of solvable Lie algebras that split over each term in their derived series. 

\begin{lemma}\label{l:hom2} ${\mathcal S_{\infty}}$ is a homomorph.
\end{lemma}
\begin{proof} This is straightforward.
\end{proof}
\bigskip

We have the following characterisation of ${\mathcal S}_{\infty}$.

\begin{theor}\label{t:compsplit} Let $L$ be a solvable Lie algebra of derived length $n+1$. Then the following are equivalent:
\begin{itemize}
\item[(i)] $L \in {\mathcal S}_{\infty}$;
\item[(ii)] $[L^{(i)}, L^{(i-1)}] = L^{(i)}$ for every $1 \leq i \leq n$;
\item[(iii)] the Cartan subalgebras of $L^{(i-1)}/L^{(i+1)}$ are precisely the subalgebras that are complementary to $L^{(i)}/L^{(i+1)}$ for $1 \leq i \leq n$;
\item[(iv)] the lower nilpotent series for $L$ coincides with the derived series for $L$; and
\item[(v)] every factor algebra in the lower nilpotent series for $L$ is abelian.
\end{itemize}
\end{theor}
\begin{proof} (i) $\Rightarrow$ (ii): This follows from Theorem \ref{t:split}.
\medskip

\noindent (ii) $\Rightarrow$ (iii): This also follows from Theorem \ref{t:split}.
\medskip

\noindent (iii) $\Rightarrow$ (i): It follows from (iii) and Theorem \ref{t:split} that $L$ splits over $L^{(n)}$. So we have that $L = L^{(n)} \dot{+} B$ where $B$ is a subalgebra of $L$. Clearly $B^{(n)} = 0$ and $B \cong L/L^{(n)}$, and so, by the same argument, $B$ splits over $B^{(n-1)}$, say $B = B^{(n-1)} \dot{+} D$. But then $L = L^{(n)} \dot{+} (B^{(n-1)} \dot{+} D) = L^{(n-1)} \dot{+} D$. Continuing in this way gives the desired result.
\medskip

\noindent (ii) $\Rightarrow$ (iv): We require that $L_i = L^{(i)}$ for each $i \geq 0$. We use induction. Clearly $L_0 = L^{(0)} = L$.
Suppose that $L_k = L^{(k)}$ for some $k \geq 0$. Then $L^{(k)}/L^{(k+1)}$ is abelian and so $L_{k+1} \subseteq L^{(k+1)}$. Moreover,
$L^{(k+1)} = [L^{(k)}, L^{(k+1)}] \subseteq [L_k, L_k] = L_k^2$, whence $L^{(k+1)} = [L^{(k)}, L^{(k+1)}] \subseteq [L_k, L_k^2] = L_k^3$ and a straightforward induction proof shows that $L^{(k+1)} \subseteq (L_k)^{\infty} = L_{k+1}$, and (iv) follows.
\medskip

\noindent (iv) $\Rightarrow$ (v): This is clear. 
\medskip

\noindent (v) $\Rightarrow$ (iv): We proceed as  above. We require that $L_i = L^{(i)}$ for each $i \geq 0$. We use induction. Clearly $L_0 = L^{(0)} = L$. Suppose that $L_k = L^{(k)}$ for some $k \geq 0$. Then $L^{(k)}/L^{(k+1)}$ is abelian and so $L_{k+1} \subseteq L^{(k+1)}$. Moreover, $L^{(k)}/L_{k+1} = L_{k}/L_{k+1}$ is abelian and so $L^{(k+1)} \subseteq L_{k+1}$, and (iv) follows.
\medskip

\noindent (iv) $\Rightarrow$ (ii): Suppose that (iv) holds. Then $$[L^{(i)}, L^{(i-1)}] = [L_{i}, L_{i-1}] = [(L_{i-1})^{\infty}, L_{i-1}] = (L_{i-1})^{\infty} = L_i = L^{(i)}.$$
\end{proof}
\bigskip

This yields the next result, which is analogous to \cite[Corollary 3.2]{Aalg} and \cite[Corollary 3.3]{comp}.

\begin{coro}\label{c:decomp} 
Let $L \in {\mathcal S}_{\infty}$ be a solvable Lie algebra of derived length $n+1$. Then the following hold:
\begin{itemize}
\item[(i)] $L = A_{n} \dot{+} A_{n-1} \dot{+} \ldots \dot{+} A_0$ where $A_i$ is an abelian subalgebra of $L$ for each $0 \leq i \leq n$; and
\item[(ii)] $L^{(i)} = A_{n} \dot{+} A_{n-1} \dot{+} \ldots \dot{+} A_{i}$ for each $0 \leq i \leq n$
\end{itemize}
\end{coro}
\begin{proof} By Theorem \ref{t:compsplit} there is a subalgebra $B_n$ of $L$ such that $L = L^{(n)} \dot{+} B_n$. Put $A_n = L^{(n)}$. Similarly $B_n \cong L/L^{(n)}$ satisfies the conditions of Theorem \ref{t:compsplit}, by Lemma \ref{l:hom}, so $B_n = A_{n-1} \dot{+} B_{n-1}$ where $A_{n-1} = B_n^{(n-1)}$. Continuing in this way we get (i). A straightforward induction proof shows (ii).
\end{proof}
\bigskip

Our final result in this section is a generalisation of \cite[Theorem3.3]{Aalg}.

\begin{lemma}\label{l:int} If $L \in {\mathcal S_{\infty}}$ then $Z(L^{(i)}) \cap L^{(i+1)} = 0$.
\end {lemma}
\begin{proof} Clearly we can assume that $L^{(i+1)} \neq {0}$. We have that $L^{(i)} = L^{(i+1)} \dot{+} A_i$ from Corollary \ref{c:decomp}. Let $x = a_n + \ldots + a_{i+1} \in Z(L^{(i)}) \cap L^{(i+1)}$ where $a_j \in A_j$ for $i+1 \leq j \leq n$. Then $0 = [x, A_{i+1} + A_i]$, so $[a_{i+1}, A_{i+1} + A_i] \subseteq (A_{i+1} + A_i) \cap L^{(i+2)} = 0$. It follows that $a_{i+1} \in Z_{A_{i+1} + A_i}(A_i) = A_i$ since $A_i$ is complementary to $A_{i+1} \cong L^{(i+1)}/L^{(i+2)}$ in $A_{i+1} + A_i \cong L^{(i)}/L^{(i+2)}$ and so is a Cartan subalgebra of $A_{i+1} + A_i$, by Theorem \ref{t:compsplit}. Hence $a_{i+1} = 0$. But now $0 = [x, A_{i+2} + A_{i+1}]$, so $[a_{i+2}, A_{i+2} + A_{i+1}] \subseteq (A_{i+2} + A_{i+1}) \cap L^{(i+3)} = 0$, giving $a_{i+2} = 0$ as before. Continuing in this way shows that $x = 0$.
\end{proof}

\section{Quasi-$A$-algebras}
A solvable Lie algebra $L$ will be called a {\em quasi-$A$-algebra} if $N(L/B)$ is abelian for every ideal $B$ of $L$. We denote the class of all such algebras by $q{\mathcal A}$. Clearly $A$-algebras are quasi-$A$-algebras. The following lemma shows that so are solvable complemented Lie algebras.

\begin{lemma}\label{l:comp} Let $L$ be a solvable complemented Lie algebra. Then $L \in q{\mathcal A}$.
\end{lemma}
\begin{proof} Since $L$ is complemented, all factor algebras of $L$ are $\phi$-free, by \cite[Theorem 2.1 (iii)]{comp}, and so have abelian nilradical by \cite[Theorem 7.4]{frat}.
\end{proof} 

\begin{lemma}\label{l:hom3} $q{\mathcal A}$ is a homomorph.
\end{lemma}
\begin{proof} This is clear from the definition.
\end{proof}

\begin{lemma}\label{l:sub} $q{\mathcal A} \subseteq {\mathcal S_{\infty}}$.
\end{lemma}
\begin{proof} Let $L \in q{\mathcal A}$ and let $L_i$ be the $i^{th}$ term of the lower nilpotent series for $L$. Then $L_i/L_{i+1}$ is a nilpotent ideal of $L/L_{i+1}$ and so is abelian. It follows from Theorem \ref{t:compsplit} that $L \in {\mathcal S_{\infty}}$.
\end{proof}
\bigskip

The following result can be proved in the same way as \cite[Lemma 3.5]{comp}. It is also a generalisation of \cite[Lemma 3.4]{Aalg}.

\begin{lemma}\label{l:ideal} 
Let $L \in q{\mathcal A}$ have derived length $\leq n+1$, and suppose that $L = B \dot{+} C$ where $B = L^{(n)}$ and $C$ is a subalgebra of $L$. If $D$ is an ideal of $L$ then $D = (B \cap D) \dot{+} (C \cap D)$.
\end{lemma}
\bigskip

The next result generalises \cite[Theorem 3.5]{Aalg} and  \cite[Theorem 3.6]{comp}.

\begin{theor}\label{t:nz}
Let $L \in q{\mathcal A}$ be a solvable Lie algebra of derived length $\leq n+1$, let $K$ be an ideal of $L$, and $A$ be a minimal ideal of $L$. Put $N_K/K = N(L/K)$, the nilradical of $L/K$, and $N = N_{\{0\}}$. Then the following hold: 
\begin{itemize}
\item[(i)] $K = (K \cap A_n) \dot{+} (K \cap A_{n-1}) \dot{+} \ldots \dot{+} (K \cap A_0)$ for all ideals $K$ of $L$;
\item[(ii)] $N_K = A_n \oplus (N_K \cap A_{n-1}) \oplus \ldots \oplus (N_K \cap A_0)$; 
\item[(iii)] $Z_{L^{(i)}}((L^{(i)} +K)/K) = N_K \cap A_i$ for each $0 \leq i \leq n$; and
\item[(iv)] $A \subseteq N \cap A_i$ for some $0 \leq i \leq n$.  
\end{itemize}
\end{theor}
\begin{proof} The proofs of (i), (ii) and (iv) are the same as those in \cite[Theorem 3.6 (i), (ii) and (iv)]{comp}.
\par

\noindent (iii) Since the hypothesis is factor algebra closed we can assume that $K = 0$; that is, it suffices to show that $Z(L^{(i)}) = N \cap A_i$. This follows as in \cite[Theorem 3.6 (iii)]{comp}. 
\end{proof}
\bigskip

The following is a straightforward consequence of the above result.

\begin{coro}\label{c:nz} Let $L \in {\mathcal S_{\infty}}$ be a solvable Lie algebra of derived length $\leq n+1$. Then $L \in q{\mathcal A}$ if and only if the following hold for every ideal $K$ of $L$:
\begin{itemize}
\item[(i)] $N_K = A_n \oplus (N_K \cap A_{n-1}) \oplus \ldots \oplus (N_K \cap A_0)$; and 
\item[(ii)] $Z_{L^{(i)}}((L^{(i)} +K)/K) = N_K \cap A_i$ for each $0 \leq i \leq n$,
\end{itemize}
where $N_K/K = N(L/K)$, the nilradical of $L/K$.
\end{coro}

\section{Completely solvable algebras}
A Lie algebra $L$ is called {\it completely solvable} if $L^2$ is nilpotent. Over a field of characteristic zero every solvable Lie algebra is completely solvable. 

\begin{lemma}\label{l:one} Let $L \in {\mathcal S}_n$, where $n \geq 2$, be a completely solvable Lie algebra. Then $L^{(n)} = 0$.
\end{lemma}
\begin{proof} We have $L = L^{(n)} \dot{+} B$ for some subalgebra $B$ of $L$. Since $L^2$ is nilpotent and $n \geq 2$, $L^{(n)} \subseteq L^{(2)} \subseteq \phi(L)$, by \cite[Corollary 4.2 and section 5]{frat}. Hence $L = \phi(L) + B = B$ and $L^{(n)} = 0$.
\end{proof}
\bigskip

Clearly then completely solvable Lie algebras $L \in {\mathcal S_2}$ are metabelian (that is, $L^{(2)} = 0$) so we get stronger versions of Theorems \ref{t:compsplit} and \ref{t:nz}.

\begin{theor}\label{t:cssplit}
Let $L$ be a completely solvable Lie algebra. Then the following are equivalent:
\begin{itemize}
\item[(i)] $L \in {\mathcal S_1} \cup {\mathcal S_2} = {\mathcal S_{\infty}}$;
\item[(ii)] $L^{(2)} = 0$ and $[L,L^2] = L^2$; and
\item[(iii)] $L^{(2)} = 0$ and the Cartan subalgebras of $L$ are precisely the subalgebras that are complementary to $L^2$.
\end{itemize} 
\end{theor}
\begin{proof} This is straightforward.
\end{proof}
\bigskip

Note that we can't replace ${\mathcal S_2}$ by ${\mathcal S_n}$ for any $n \neq 2$ in the above result, as the following example shows.

\begin{ex} Let $L$ be the four-dimensional Lie algebra spanned by $e_1$, $e_2$, $e_3$, $e_4$ where $[e_1,e_2] = e_3$, $[e_1,e_3] = e_2$ and $[e_2,e_3] = e_4$. Then $L^2 = Fe_2 + Fe_3 + Fe_4$ is nilpotent, $L^{(2)} = Fe_4$ and $L \in {\mathcal S}_n$ for $n \neq 2$, but $L \notin {\mathcal S}_2$. Moreover, $L^{(2)} = Z(L)$ and $Fe_1 + Fe_4$ is a Cartan subalgebra of $L$. 
\end{ex}

\begin{theor}\label{t:csnz}
Let $L \in q{\mathcal A}$ be a completely solvable Lie algebra with $L = L^2 \dot{+} B$ where $B$ is a subalgebra of $L$, let $K$ be an ideal of $L$ and let $A$ be a minimal ideal of $L$. Put $N_K/K = N(L/K)$, the nilradical of $L/K$. Then 
\begin{itemize}
\item[(i)] $K = K \cap L^2 \dot{+} K \cap B$;
\item[(ii)] $N_K = L^2 \oplus Z_B(L/K)$; and
\item[(iii)] $A \subseteq L^2$ and $[A,L] = A$, or $A \subseteq B$ and $A \subseteq Z(L)$ (in which case $\dim A = 1$).
\end{itemize} 
\end{theor}
\begin{proof} This follows easily from Theorem \ref{t:nz}.
\end{proof}

\begin{coro}\label{c:csnz} Let $L \in {\mathcal S_{\infty}}$ be a completely solvable Lie algebra with $L = L^2 \dot{+} B$ where $B$ is a subalgebra of $L$. Then $L \in q{\mathcal A}$ if and only if $N_K = L^2 \oplus Z_B(L/K)$ for all ideals $K$ of $L$, where $N_K/K = N(L/K)$, the nilradical of $L/K$.
\end{coro}

\begin{theor}\label{t:aalg} Let $L$ be a completely solvable Lie algebra. Then $L \in q{\mathcal A}$ if and only if $L$ is an $A$-algebra.
\end{theor}
\begin{proof} Suppose first that $L \in q{\mathcal A}$ and let $U$ be a maximal nilpotent subalgebra of $L$. Let $L = {\mathcal L}_0 \dot{+} {\mathcal L}_1$ be the Fitting decomposition of $L$ relative to ad\,$U$. Then ${\mathcal L}_1 \subseteq L^2$, so ${\mathcal L}_1$ is an abelian ideal of $L$ and ${\mathcal L}_0 \cong L/{\mathcal L}_1 \in q{\mathcal A}$. We thus have ${\mathcal L}_0 = {\mathcal L}_0^2 \dot{+} C$ where $C$ is an abelian subalgebra of ${\mathcal L}_0$. Moreover, $L^2 = ({\mathcal L}_0 \cap L^2)\dot{+} {\mathcal L}_1$ and 
$$[L,{\mathcal L}_0 \cap L^2] = [{\mathcal L}_0 + {\mathcal L}_1,{\mathcal L}_0 \cap L^2] \subseteq ({\mathcal L}_0 \cap L^2) + L^{(2)} = {\mathcal L}_0 \cap L^2,$$ 
so ${\mathcal L}_0 \cap L^2$ is an ideal of $L$. It follows that $U + ({\mathcal L}_0 \cap L^2)$ is a nilpotent subalgebra of $L$ and so ${\mathcal L}_0 \cap L^2 \subseteq U \cap L^2$. The reverse inclusion is clear, whence ${\mathcal L}_0 \cap L^2 = U \cap L^2$.
\par

Now $L^2 = {\mathcal L}_0^2 \oplus [{\mathcal L}_0,{\mathcal L}_1]$, so ${\mathcal L}_0^2 = {\mathcal L}_0 \cap L^2 \subseteq U$. Thus $U = {\mathcal L}_0^2 + U \cap C$ is an ideal of ${\mathcal L}_0$ and so is in the nilradical of ${\mathcal L}_0$. It follows that $U$ is abelian and $L$ is an $A$-algebra.
\par

The converse is clear.
\end{proof}
\bigskip

The above result does not hold for all solvable Lie algebras, as complemented Lie algebras belong to $q{\mathcal A}$ and these are not all $A$-algebras (see, \cite[Example 3.9]{comp}).
\par

The classes ${\mathcal S_{\infty}}$ and $q{\mathcal A}$ are different even in the case of completely solvable Lie algebras, as the following example shows.

\begin{ex}\label{e:1} Let $L = Fx_1 + Fx_2 + Fx_3 + Fx_4$ with $[x_4,x_1] = x_1$, $[x_4,x_2] = x_2$ and $[x_3,x_1] = x_2$, other products being zero. Then $L^2 = Fx_1 + Fx_2$, so $L \in {\mathcal S_{\infty}}$. However, $N = L^2 + Fx_3$, which is not abelian, and so $L \notin q{\mathcal A}$. Notice that $\phi(L) = Fx_2$.
\end{ex}

The two classes do coincide for $\phi$-free algebras as the next result shows. The {\em abelian socle}, Asoc$L$, of $L$ is the sum of the minimal abelian ideals of $L$.

\begin{theor}\label{t:phi}
Let $L$ be a completely solvable Lie algebra. Then  
\begin{itemize}
\item[(i)] $L$ is $\phi$-free only if $L \in q{\mathcal A}$; and
\item[(ii)] $L$ is $\phi$-free if and only if $L^2 \subseteq $ Asoc$L$ and $L \in {\mathcal S_{\infty}}$.
\end{itemize}
\end{theor}
\begin{proof} (i) Suppose that $L$ is $\phi$-free. Then $L$ is complemented by \cite[Theorem 2.2]{comp} and so $L \in q{\mathcal A}$ by Lemma \ref{l:comp}. 
\par
(ii) Suppose first that $L$ is $\phi$-free. Then $L^2 \subseteq N = $ Asoc$L$, by \cite[Theorem 7.4]{frat}. Also $L \in {\mathcal S_{\infty}}$ by \cite[Lemma 7.2]{frat}. 
\par
So suppose now that $L^2 \subseteq $Asoc$L$ and $L \in {\mathcal S_{\infty}}$. Then $L$ splits over Asoc$L$ by Theorem \ref{t:split}. But now $L$ is $\phi$-free by \cite[Theorem 7.3]{frat}.
\end{proof}
\bigskip

Part (i) of the above result does not hold if $L$ is not completely solvable, as the following example shows.

\begin{ex}\label{e:2} Let $L$ be as in Example \ref{e:1} above, considered over a field of characteristic $p$, let $B$ be a faithful completely reducible $L$-module and put $X = B \dot{+} L$ where $B^2 = 0$ and $L$ acts on $B$ under the given $L$-module action. Then $B \subseteq$ Asoc$X$, since $B$ is completely reducible as an $X$-module. Furthermore, $N(X) \subseteq C_X(B) = B$ since $B$ is faithful. It follows that $X$ splits over Asoc$X = N(X) = B$, and hence that $\phi(X) = 0$, by \cite[Theorem 7.3]{frat}. Moreover, $X^2 = B \dot{+} Fx_1 + Fx_2$ and $X^{(2)} = B$, so $X \in {\mathcal S_{\infty}}$. However, $X \notin q{\mathcal A}$ since $L \notin q{\mathcal A}$.
\end{ex} 

Notice further that if $X$ is as in Example \ref{e:2} then $X \in {\mathcal S_{\infty}}$ and (ii) and (iii) of Theorem \ref{t:nz} hold with $K = 0$, but $X \notin q{\mathcal A}$. 

\section{$\phi$-free $q{\mathcal A}$-algebras}
Following on fron Theorem \ref{t:phi}, in this section we seek a characterisation of $\phi$-free $q{\mathcal A}$-algebras. A Lie algebra $L$ is called {\em monolithic} if it has a unique minimal ideal $W$, the {\em monolith} of $L$. Then the following result can be proved in the same way as \cite[Theorem 5.1]{Aalg}.

\begin{theor}\label{t:mon}
Let $L \in q{\mathcal A}$ be a monolithic solvable Lie algebra of derived length $n+1$ with monolith $W$. Then, with the same notation as Corollary \ref{c:decomp}, 
\begin{itemize}
\item[(i)] $W$ is abelian;
\item[(ii)] $Z(L) = 0$ and $[L,W] = W$; 
\item[(iii)] $N = A_n = L^{(n)}$; 
\item[(iv)] $N = Z_L(W)$; and
\item[(v)] $L$ is $\phi$-free if and only if $W = N$.
\end{itemize} 
\end{theor}

\begin{coro}\label{c:csmon} Let $L \in q{\mathcal A}$ be a monolithic supersolvable Lie algebra. Then $L = L^2 \dot{+} Fx$ where $($ad\,$x)\mid_{L^2}$ is triangulable.
\end{coro}
\begin{proof} We have $L = L^2 \dot{+} B$ for some abelian subalgebra $B$ of $L$. Also $\dim W = 1$ and so $\dim L/L^2 = \dim L/N = \dim L/Z_L(W) \leq 1$. It follows that $L = L^2 \dot{+} Fx$, and $($ad\,$x)\mid_{L^2}$ is triangulable since $L$ is supersolvable.
\end{proof}

\begin{coro}\label{c:phifreemon} Let $L \in q{\mathcal A}$ be a $\phi$-free monolithic solvable Lie algebra of index $n+1$. Then $L = L^{(n)} \rtimes B$ is the semidirect product of $L^{(n)}$ and a solvable Lie algebra $B \in q{\mathcal A}$ of index $n$.
\end{coro}
\begin{proof} Simply note that $L^{(n)} = Z_L(L^{(n)})$, by Theorem \ref{t:mon}, and $L = L^{(n)} \dot{+} B$ where $B \cong L/Z_L(L^{(n)})$, which is isomorphic to an irreducible subalgebra of $gl(L^{(n)})$.
\end{proof}
\bigskip

The next result gives us a way of constructing a $\phi$-free monolithic $q{\mathcal A}$ algebra of index $n+1$ from a $q{\mathcal A}$ algebra of index $n$. 

\begin{theor}\label{t:phifreemon} Let $B \in q{\mathcal A}$ be a solvable Lie algebra and let $A$ be an irreducible $B$-module. Then $L = A \rtimes B$ is a solvable Lie algebra with $L \in q{\mathcal A}$. Moreover, if $A$ is faithful and $B$ has index $n$, then $L$ is $\phi$-free, monolithic and has index $n+1$.
\end{theor}
\begin{proof} Let $K$ be an ideal of $L$, and let $N_K/K = N(L/K)$. Suppose first that $A \subseteq K$. Then $N_K/K \cong (N_K/A)/(K/A)$ which yields that $(N_K)^2 \subseteq K$ since $L/A \in q{\mathcal A}$.
\par

So suppose now that $A \not \subseteq K$. Then $[N_K,A] = A$ implies that $A \subseteq (N_K)^r$ for every $r \geq 2$, and hence that $A \subseteq K$, a contradiction. It follows that $[N_K,A] = 0$, from the irreducibility of $A$. Clearly $(N_K \cap B)^2 \subseteq K \cap B$, since $B \in q{\mathcal A}$. Hence $(N_K)^2 = (A + N_K \cap B)^2 \subseteq K$. We therefore have that $L \in q{\mathcal A}$.
\par

Next assume that $A$ is faithful and $B$ has index $n$.  Clearly $L^{(i)} = A + B^{(i)}$ for $1 \leq i \leq n$, so $L$ has index $n+1$. If $K$ is an ideal of $L$, then $A \cap K = 0$ or $A \cap K = A$ from the irreducibility of $A$. The former implies that $[A,K] = 0$, so $K \subseteq Z_L(A) = A$, a contradiction. It follows that $A \subseteq K$ and $A$ is the monolith of $L$. If $\phi(L)$ is non-zero then $A \subseteq \phi(L)$ which yields $L = \phi(L) + B = B$, a contradiction. Hence $L$ is $\phi$-free.
\end{proof}
\bigskip

Our final result gives the characterisation of $\phi$-free $q{\mathcal A}$ algebras that we are seeking.

\begin{theor} Let $L$ be a solvable Lie algebra of with nilradical $N$. Then the following are equivalent:
\begin{itemize}
\item[(i)] $L \in q{\mathcal A}$ and is $\phi$-free; and
\item[(ii)] $N$ is abelian and $L = N \rtimes B$, where $B \in q{\mathcal A}$ and $N$ is a completely reducible $B$-module.
\end{itemize}
Furthermore, $L$ has index $n+1$ if $B$ has index $n$ and $N$ has a faithful irreducible submodule.
\end{theor}
\begin{proof} (i) $\Rightarrow$ (ii): Assume (i). Since $L$ is $\phi$-free, $N =$ Asoc\,$L$ and $L = N \dot{+} B$, where $B$ is a subalgebra of $L$, by \cite[Theorems 7.3 and 7.4]{frat}. Moreover, $Z_L(N) = N$, so $B \cong L/Z_L(N)$ is isomorphic to a subalgebra of $gl(N)$ and $N$ is faithful. Also, $N$ is completely reducible, since $N =$ Asoc\,$L$. Finally, if $A$ is a faithful irreducible submodule of $N$, then $L^{(i)} \supseteq A + B^{(i)}$ for $1 \leq i \leq n$, and so $L$ has index one greater than that of $B$.
\par

\noindent (ii) $\Rightarrow$ (i): So now suppose that (ii) holds. Since $N$ is abelian and completely reducible, $N =$ Asoc\,$L$ and $L$ is $\phi$-free, by \cite[Theorem 7.3]{frat}. 
\par

Let Asoc\,$L = A_1 \oplus \ldots \oplus A_r$, where $A_i$ is a minimal ideal of $L$ and put $L_0 = B$, $L_i = A_1 \oplus \ldots \oplus A_i + B$ for $1 \leq i \leq r$. Then $A_i$ is an irreducible  $L_{i-1}$-module for $1 \leq i \leq r$. It follows from Theorem \ref{t:phifreemon} and a simple induction argument that $L \in q{\mathcal A}$. 
\par

If $B$ has index $n$ and $N$ has a faithful irreducible submodule, then $L$ has index $n+1$ as above.
\end{proof}

\end{document}